\def\IR{{\Bbb R}} 
 
\def\IS{{\Bbb S}} 
\def\IK{{\Bbb K}}

\def\IC{\Bbb C} 
\def\ID{{\Bbb D}}
\def\IA{{\Bbb A}}

\def\zbar{{\overline{z}}}

\def\Mod{{\rm mod}\, }

\documentclass[12pt]{article} 

\usepackage[psamsfonts]{amssymb} 
\usepackage{amsfonts,amsmath} 
\usepackage{amsthm} 
\usepackage{graphicx}
 
\newcounter{minutes}
\divide\time by 60
\newcounter{hours}
\multiply\time by 60 \addtocounter{minutes}{-\time}

\newtheorem{theorem}{Theorem}
\newtheorem{lemma}{Lemma}
\newtheorem{corollary}{Corollary}

\usepackage{amsmath, amssymb, amscd, amsxtra, txfonts}

\title{Distortion and topology}
\author{Riku Kl\'en \thanks{Research supported by the Marsden Fund.} \and Gaven Martin \footnotemark[1] \thanks{Thanks to Aalto University and the University of Helsinki where this research was initiated and IPAM (UCLA) and ICERM (Brown) where it was partially supported.}} 
\date{} 
\begin{document} 

\maketitle 

\def\thefootnote{}
\footnotetext{
\texttt{\tiny File:~\jobname .tex,
          printed: \number\year-\number\month-\number\day,
          \thehours.\ifnum\theminutes<10{0}\fi\theminutes}
}
\makeatletter\def\thefootnote{\@arabic\c@footnote}\makeatother

\begin{abstract}  For a self mapping $f:\ID\to \ID$ of the unit disk in $\IC$ which has finite distortion,  we give a separation condition on the components of the set where the distortion is large - say greater than a given constant - which implies that $f$ extends homeomorphically and quasisymetrically to the boundary $\IS$ and thus $f$ shares its boundary values with a quasiconformal mapping whose distortion can be explicitly estimated in terms of the data.  This result holds more generally.    This condition,  {\em uniformly separated in modulus},  allows the set where the distortion is large to accumulate densely on the boundary but does not allow a component to run out to the boundary.  The lift of a Jordan domain in a Riemann surface to its universal cover $\ID$ is always uniformly separated in modulus and this allows us to apply these results in the theory of Riemann surfaces to identify an interesting link between the support of the high distortion of a map and topology of the surface - again with explicit and good estimates.  As part of our investigations we study mappings $\varphi:\IS\to\IS$ which are the germs of a conformal mapping and give good bounds on the distortion of a quasiconformal extension of $\varphi$.  We extend these results to the germs of quasisymmetric mappings.  These appear of independent interest and identify new geometric invariants.
\end{abstract}

\section{Introduction}   

The theory of mappings of finite distortion has been developed over the last couple of decades to extend the classical theory of quasiconformal mappings in new directions to build stronger linkages with the calculus of variations.  The underlying elliptic PDEs of quasiconformal mappings - in particular  Beltrami equations - are replaced by their  degenerate elliptic counterparts.  A recent thorough account of the two dimensional theory is given in \cite{AIM,IMbelt} (and the references therein) while the higher dimensional theory is accounted for in \cite{IMbook}.  Recent problems seek to study various minimisation problems for integral means of distortion, see \cite{AIM1,AIMem,Martin3} for the $L^1$ case and \cite{IMO} for the $L^p$ case.  An eventual aim is to develop an $L^p$-Teichm\"uller theory.

Without the a priori bounds of the theory of quasiconformal mappings,  sequences of mappings of finite distortion may degenerate quite badly.  

 Here we use these ideas to study the boundary values of self homeomorphisms of finite distortion defined on the disk and give applications in the theory of Riemann surfaces.   One of the great virtues of our approach is in achieving very explicit and clean estimates.

\section{Germs of quasisymmetric mappings}

For $0\leq r< R$ we define the annulus  
\begin{equation*}
\IA(r,R) = \{z: r < |z| < R \}.
\end{equation*}
Let $U\subset \ID$ be an open doubly connected (a {\em ring}) subset of $\ID$ with $\IS\subset \partial U$. Let $\mu:U\to\ID$ with $\|\mu\|_{L^\infty(U)}=k_{\mathfrak{g}}<1$.  Then we can solve the Beltrami equation  
\begin{equation}\label{be}
g_\zbar = \mu(z)\,  g_z, \hskip20pt \mbox{almost every $z\in U$}
\end{equation}  
for a quasiconformal homeomorphism $g:U\to \IC$,  uniquely up to conformal mappings of $g(U)$.  Now
$g(U)$ is also doubly-connected and so conformally equivalent to the annulus $\IA(r,1)$ when $\log 1/r = {\rm mod}(g(U))$ -  the conformal modulus.   That is there is a conformal mapping $\psi:g(U)\to\IA(r,1)$,  which is unique up to a rotation.  Then $\psi\circ g:U\to \IA(r,1)$ and the Carath\'eodory extension (or reflection) principle shows this quasiconformal map extends as a quasiconformal mapping $\widetilde{\psi\circ g}:U\cup \IS \cup U^* \mapsto  \IA(r,1/r)$,  where $U^*=\{1/z:z\in U\}$. In this way the pair  $(U,\mu)$ determines  a unique quasisymmetric mapping $g_0=\widetilde{\psi\circ g}\big|\IS:\IS\to\IS$ with $g_0(1)=1$.   We call the pair $\mathfrak{g}=(U,\mu)$ a germ of the quasisymmetric homeomorphism $g_0:\IS\to\IS$.  

\begin{figure}[ht!]
\begin{center}
  \includegraphics[width=.7 \textwidth]{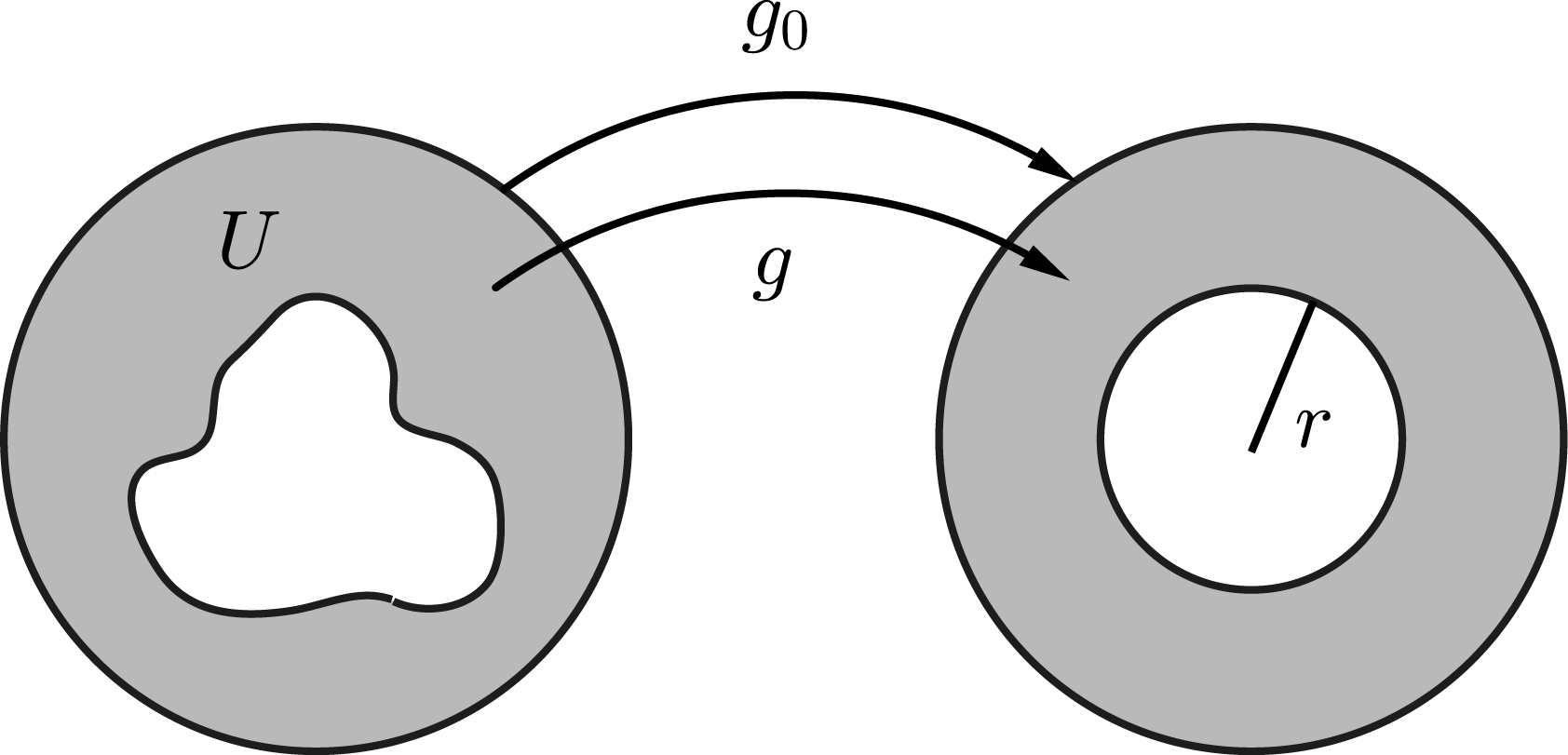}
  \caption{Conformal mapping $g \colon U \to \IA(r,1)$, $g_0 \colon \IS \to \IS$ and ${\rm mod} (U) = {\rm mod} (\IA(r,1)) = \log \frac1r$.}
\end{center}
\end{figure}

There are,  of course,  natural structures on the set of germs via composition and restriction.  Note that it is only under very special circumstances that for $V\subset U$,  the maps $(U,\mu)\mapsto g_U$ and $(V,\mu|V)\mapsto g_v$ of germs give the same quasisymmetric mapping - even when $\mu\equiv 0$.  For instance $(\IA(r,1),0)$ is a germ of the identity map for all $r>0$.  For $\mu\not\equiv 0 $, uniqueness of solutions to the Beltrami equation up to conformal equivalence shows that all examples are modelled similarly. Further,  if $U\neq \IA(r,1)$ for any $r$,  then $(U,0)\not\mapsto identity$.  Another example is furnished by the following easy lemma.  

\begin{lemma} \label{lemma1} Let $\mathfrak{g}=(U,0)$ be the germ of the quasisymmetric mapping $g_0$ and suppose that the inner boundary component of $U$ is a circle.  Then $g_0=\phi|\IS$ for a M\"obius transformation $\phi$ of the disk.
\end{lemma}

Now,  as a quasisymmetric homeomorphism $\IS\to\IS$,  the map $g_0$ admits a quasiconformal extension $G:\ID\to\ID$ with $G|\IS=g_0$.  There are many such extensions,  the most well-known are the Ahlfors-Beurling \cite{AB} or the Douady-Earle extensions \cite{DE},  but there are others,  see \cite{AIM}.  A problem with these extensions is the quite poor effective bounds one achieves between the quasisymmetry constants and the quasiconformal distortion. Here we seek effective bounds on $\|\mu_G\|_{L^{\infty}(\ID)}$ in terms of $k_{\mathfrak{g}}$ and the conformal modulus of $U$.  The problem is quite nontrivial even in the case of germs of the form $\mathfrak{g}=(U,0)$.  Here $\mathfrak{g}\mapsto g_0$ is a real analytic diffeomorphism of $\IS$.  A germ $\mathfrak{g}=(U,\mu)$ naturally defines three quantities; ${\rm mod}(U)$, $k_{\mathfrak{g}} = \|\mu\|_{L^{\infty}(U)} $, and 
\begin{eqnarray} \label{2}
m_{\mathfrak{g}} & = &  {\rm mod}(g(U)), \hskip10pt \mbox{ where $g$ solves (\ref{be}) }.
\end{eqnarray}
Since $g$ is $\frac{1+k_{\mathfrak{g}}}{1-k_{\mathfrak{g}}}$--quasiconformal we have the estimate
\begin{equation*}
 \frac{1-k_{\mathfrak{g}}}{1+k_{\mathfrak{g}}}  \;  {\rm mod}(U)\leq m_{\mathfrak{g}} \leq  \frac{1+k_{\mathfrak{g}}}{1-k_{\mathfrak{g}}}  \;  {\rm mod}(U).
\end{equation*}

\section{Three problems}
\noindent{\bf Problem 1.} {\em Bound the maximal distortion of a quasiconformal extension of the quasisymmetric mapping $g_0$ defined by the germ $\mathfrak{g}=(U,0)$. }

\medskip

We will give two answers to this problem.  Surprisingly, we will see that the answer seems to depend on the ``roundness'' of the inner boundary component of $U$ and not more subtle invariants - this is suggested by Lemma \ref{lemma1}.   What  is less clear is that this inner boundary component may have positive measure,  yet $g_0$   admits a $(1+\epsilon)$--quasiconformal extension if it is ``nearly circular''.

\medskip

Then we use this results to obtain information about the more general question:

\medskip

\noindent{\bf Problem 2.} {\em Bound the maximal distortion of a quasiconformal extension of the quasisymmetric mapping $g_0$ defined by the germ $\mathfrak{g}=(U,\mu)$ in terms of ${\rm mod}(U)$ and $k_{\mathfrak{g}}$.}

\medskip

The solution to Problems 1 and 2 are clean,  with effective estimates and enable us to consider the more general problem. Precise definitions are given below.

\medskip

\noindent{\bf Problem 3.} {\em Suppose $f:\ID \to\ID$ is a homeomorphism of finite distortion with Beltrami coefficient $\mu_f = f_\zbar/f_z$.  Give conditions on a set $E$ to have the following property:  If  
\[ \|\mu_f\|_{L^\infty(\ID\setminus E)} \leq k <1, \]  then $f$ extends to a homeomorphism $f:\IS \to\IS$ which is quasisymmetric and admits a quasiconformal extension to $\ID$ for which the maximal distortion depends effectively on $k$ and the geometry of $E$.  }

\medskip

The most interesting case here is when $E$ does not have compact closure in $\ID$ (for in that case the solution to Problem 2 will apply).  The condition we will find on $E$,  being {\em uniformly separated in modulus},  is similar to being `porous',  but implies $E$ cannot have components running to the boundary.  This is necessary as we will see.  If $\Sigma$ is a Riemann surface, and $\Omega$ is a Jordan domain in $\Sigma$,  then the lifts of $\Omega$ to the universal cover $\ID$ will satisfy the uniformly separated in modulus condition. Then using all these results we will  show - with explicit estimates - that distortion and topology must interact for degeneration to occur in sequences of Riemann surfaces.  As an example suppose we have a base Riemann surface $\Sigma_0$ and a sequence of quasiconformally equivalent surfaces converging to a limit surface geometrically - $\Sigma_{i} \to \Sigma_{\infty}$ with quasiconformal maps $f_i:\Sigma_0\to\Sigma_i$.  With no a priori bounds on the distortion, the surface $\Sigma_\infty$ may have a different topological type.  However,  suppose that $\Omega\subset \Sigma$  is a Jordan domain  and that
\begin{equation*}
\sup_{i} \; \| \; \mu_{f_i} \; \|_{\;L^\infty(\Sigma\setminus\Omega)} \leq k< 1.
\end{equation*}
  Then we will show $\Sigma_\infty$ is quasiconformally equivalent to $\Sigma$.   The hypothesis that $\overline{ \Omega}$ is simply connected is essential as examples given by shrinking a simple closed loop on a particular surface show.  
  
\medskip

Problems 1 \& 2 have quite nice solutions in Theorem \ref{mainthm2} and Theorem \ref{thm5} below.  In particular,  in both cases we achieve the bound 
\[ \left(1+\frac{10}{m_\mathfrak{g}} \right)\; K,  \hskip20pt  K=\frac{1+k}{1-k} \]
with $K=1$ for Problem 1.  Notice this bound has good behaviour as ${\rm mod}(U)\to\infty$. The bound $10$ is not sharp,  but examples show that this bound has the correct structure for ${\rm mod}(U)$ small and the number $10$ cannot be replaced by any constant less than $1$. To achieve this bound in Problem 1, we  use the solution to Problem 2  which relies on a more complex solution to Problem 1.

\medskip

We now recall  two of the basic notions we will need for this paper.  
\subsection{Mappings of finite distortion}
A homeomorphism $f:\Omega\subset \IC \to \IC$ defined on a domain $\Omega$ and lying in the Sobolev class $f\in W^{1,1}_{loc}(\Omega,\IC)$ of functions with locally integrable first derivatives is said to have {\em finite distortion} if there is a  distortion  function $K(z,f)$, finite almost everywhere, so that $f$ satisfies the {\em distortion inequality} 
\begin{equation*}
|Df(z)|^2 \leq K(z,f) \; J(z,f), \hskip20pt\mbox{almost everywhere in $\Omega$.}
\end{equation*}
Here  $Df(z)$ is the Jacobian matrix and $J(z,f)$ its determinant.  If $K(z,f)\in L^\infty(\Omega)$ and $K_f=\|K(z,f)\|_\infty$,  then $f$ is $K_f$--quasiconformal.  The basic theory of mappings of finite distortion has been developed in recent years,  and the two dimensional aspects are  described in \cite{AIM},  though there is much interesting recent work.

If the mapping $f$ has finite distortion,  then $f$ has a Beltrami coefficient $\mu:\Omega \to \ID$ and $f$ solves the degenerate Beltrami equation 
 \begin{equation*}
 f_\zbar(z) = \mu(z)\;  f_z(z) , \hskip20pt\mbox{almost every $z\in\Omega$.}
 \end{equation*}
The relationship between $K(z,f)$ and $\mu(z)$ is simply $K(z,f)=\frac{1+|\mu(z)|}{1-|\mu(z)|}$ so that if $\|\, \mu\, \|_\infty = k <1$,  then $f$ is quasiconformal.  Be aware that existence and uniqueness questions for solutions to the Beltrami equation in this degenerate setting are quite involved and such classical results as Stoilow factorisation may well fail without additional regularity assumptions.

 \section{Quasiconformal extension} 

In this section we consider the first problem we raised.  Let $\mathfrak{g}=(U,0)$.  From this germ $\mathfrak{g}$,  as discussed above, we can construct a conformal mapping 
\begin{equation}\label{varphidef}
\varphi:V\to  \IA(r_0,1/r_0), \hskip20pt r_0=e^{-{\rm mod}(U)}
\end{equation}
defined on $V=U\cup \IS \cup U^*$,  and $g_0=\varphi|\IS:\IS\to\IS$ is quasisymmetric.  Now $g_0$ will be a real analytic diffeomorphism and so we can consider extensions whose distortion bounds depend on derivatives.  
 
 \subsection{Radial extension}
 
 Let us first discuss an obvious method of extension of bilipschitz homeomorphisms.
 
\begin{lemma}\label{ext1}
 Let $g_0(e^{i\theta})=e^{if(\theta)}:\IS\to \IS$ be a bilipschitz homeomorphism with $f:[0,2\pi)\to[0,2\pi)$ increasing and $0< \ell \leq f'(\theta) \leq L<\infty$ almost everywhere.  Then $g_0$ admits a $K$--quasiconformal extension $G:\ID\to\ID$,  $G|\IS=g_0$, with $K\leq \max\{1/\ell, L\}$.
\end{lemma}
\noindent{\bf Proof.}  Set
$
G(re^{i\theta}) = r e^{if(\theta)}.
$ 
The regularity assumptions on $f$ imply that $G$ is a mapping of finite distortion.   We calculate
 \begin{eqnarray*}
 G_\zbar & = & \frac{e^{-i\theta}}{2}\left( \frac{\partial}{\partial r} +  \frac{i}{r} \frac{\partial}{\partial\theta}\right) \, G  = \frac{e^{-i\theta}}{2}\left( e^{if(\theta)}- f'(\theta)e^{if(\theta)} \right),  \\
 G_z & = & \frac{e^{-i\theta}}{2}\left( \frac{\partial}{\partial r} - \frac{i}{r} \frac{\partial}{\partial\theta}\right) \, G  = \frac{e^{-i\theta}}{2}\Big(e^{if(\theta)}+ f'(\theta) e^{if(\theta)} \Big) 
 \end{eqnarray*}
 and so the Beltrami coefficient of $G$ has
 \[ |\mu_G| = \frac{|1-f'(\theta)|}{|1+f'(\theta)|} \]
 and therefore the distortion 
 \[ K = \frac{1+|\mu_G|}{1-|\mu_G|} = \frac{ 1+f'(\theta)+ |1-f'(\theta)|}{1+f'(\theta) - |1-f'(\theta)|}. \]
Consequently if $f'(\theta)\leq 1$ we obtain $K = 1/f'(\theta) \leq 1 / \ell$ and if $f'(\theta)\geq 1$ we have $K = f'(\theta) \leq K$.
This completes the proof. \hfill  $\Box$

\medskip

We remark that if we define
\begin{equation*}
\IK(z,f) = \frac{1}{2}\big( K(z,f)+1/K(z,f)\big) = \frac{1+|\mu_f|^2}{1-|\mu_f|^2}
\end{equation*}
then
\begin{equation*}
\IK(z,G) =  \frac{|1-f'(\theta)|^2+|1+f'(\theta)|^2}{|1+f'(\theta)|^2-|1-f'(\theta)|^2} = \frac{1}{2} \left( f'(\theta)+\frac{1}{f'(\theta)}\right) . 
\end{equation*}

\subsection{Asymptotically conformal extension}  With the hypotheses of Lemma \ref{ext1}, if in addition,  $f'$ is Lipschitz and if we define
\begin{equation*}
G(re^{i\theta}) = r^{f'(\theta)} e^{if(\theta)},
\end{equation*}
then
\begin{eqnarray*}
 G_\zbar & = & \frac{1}{2}  e^{if(\theta)-i\theta} r^{f'(\theta)-1} \, \Big(f'(\theta) + i ( f''(\theta)\log r  + i f'(\theta))\Big) \\ & = &   \frac{-1}{2}  e^{if(\theta)-i\theta} r^{f'(\theta)-1} \,  f''(\theta)\log r,\\
 G_z& = & \frac{1}{2}  e^{if(\theta)-i\theta} r^{f'(\theta)-1} \, \Big(f'(\theta) - i ( f''(\theta)\log r  + i f'(\theta))\Big) \\ & = &   \frac{1}{2}  e^{if(\theta)-i\theta} r^{f'(\theta)-1} \, \Big(2f'(\theta) - i  f''(\theta)\log r  \Big) 
 \end{eqnarray*}
 so that
  \[ \big|\,\mu_G(re^{i\theta})\,\big| = \frac{|f''(\theta)\log r| }{|2f'(\theta) - i  f''(\theta)\log r |} = \frac{ 1}{ \sqrt{1+\left(\frac{2f'(\theta)}{f''(\theta)\log r}\right)^2}} . \]
 Notice that $|\mu_G|\to 0$ as $r\to 1$ 
 and so $G$ will be {\em asymptotically conformal}. Indeed we have the following.
 \begin{lemma}\label{lemma3}
  Let $g_0(e^{i\theta})=e^{if(\theta)}:\IS\to \IS$ be a homeomorphism with $f:[0,2\pi)\to[0,2\pi)$ increasing and 
  \begin{equation*} L_f = \Bigg\| \frac{f''(\theta)}{f'(\theta)} \Bigg\|_{L^{\infty}(\IS)} <\infty.
  \end{equation*} 
    Then $g_0$ admits an extension $G:\ID\to\ID$,  $G|\IS=g_0$ and 
  \begin{equation*}
  \frac{ |\, \mu_G(re^{i\theta})\, | }{\log 1/r } =   \frac{1}{ \sqrt{\log^2 r +\left(\frac{2f'(\theta)}{f''(\theta)}\right)^2}} \approx  \frac{1}{2} \, \frac{f''(\theta)}{f'(\theta)},  \hskip10pt \mbox{as $r\to 1$}.
  \end{equation*}
 \end{lemma}
Naturally this map is not quasiconformal  as the behaviour as $r\to 0$ is too bad,   $|\mu_G| \to 1$.  However the point is to observe that $\log 1/r = {\rm mod}(\IA(r,1))$ and so the solution to Problem 2 will determine a good global bound,  see Theorem \ref{Lf}.

 \subsection{Roundness}  The next bounds we  achieve on the maximal distortion of an extension of $g_0=\varphi|\IS$ as defined at (\ref{varphidef}) will depend on a M\"obius invariant notion of roundness of the images  $\varphi^{-1}(\IS(0,r))$,  $r> r_0 = e^{-{\rm mod}(U)}$.  We subsequently estimate these quantities geometrically.   
 
 \medskip
 
 Let $\gamma$ be a Jordan curve in $\ID$ bounding a region $\Omega_\gamma\subset\ID$.  We define
 \begin{equation*}
 {\rm mod}(\gamma) = {\rm mod}(\ID\setminus \Omega_\gamma). 
 \end{equation*}
 We then define the roundness of  $\gamma$ as follows:  For $a\in \Omega_\gamma$ define $\ell_a$ and $L_a$ by the rule
 \begin{eqnarray*}
 \log \frac{1+\ell_a}{1-\ell_a} = \inf \big\{\, \rho_{\ID}(a,z) : z\in \gamma \big\}, &&
 \log \frac{1+L_a}{1-L_a}  = \sup \big\{\, \rho_{\ID}(a,z) : z\in \gamma \big\}.
 \end{eqnarray*}
 It is immediate that if the M\"obius transformation
 \begin{equation*}
  \phi_a(z) =  \frac{z+a}{1+\bar a z},
 \end{equation*}
 then $\gamma \subset \phi_a(\IA(\ell_a,L_a))$ and this image is a hyperbolic annulus centred at $a$ and $\gamma$ winds around $a$ once.  Again,  an elementary compactness argument shows that there is a thinest such annulus as we move through possible centres $a$ - though in that case we might have $\ell_a=L_a$ and $\gamma$ is a circle with hyperbolic centre $a$.  We use this thinest annulus to measure roundness.  We have the inclusion of rings (doubly connected domains) $\ID\setminus\ID(0,L_a)\subset \phi_{-a}(\ID\setminus \Omega_\gamma)\subset \ID\setminus\ID(0,\ell_a)$ and so the monotonicity of modulus implies 
 \[  \log 1/L_a  \leq {\rm mod}(\gamma) \leq  \log 1/\ell_a   \] 
 and roundness is the ratio of these quantities:
 \begin{equation*}
 \nu(\gamma) =  \inf  \; \left\{ \max\left\{ \; \frac{\log 1/\ell }{{\rm mod}(\gamma)}, \, \frac{{\rm mod}(\gamma)}{\log 1/L }\;\right\} \; \right\} \geq 1,
 \end{equation*}
 where the infimum is taken over all $ 0\leq \ell \leq L \leq 1 $ such that there is $a\in \ID$ with $ \phi_a(\gamma)\subset \IA(\ell,L) $ and the winding number $\omega(\gamma,-a)=1$, (equivalently $\omega(\phi_a(\gamma,0)=1$).
 
\begin{figure}[ht!]
\begin{center}
  \includegraphics[width=\textwidth]{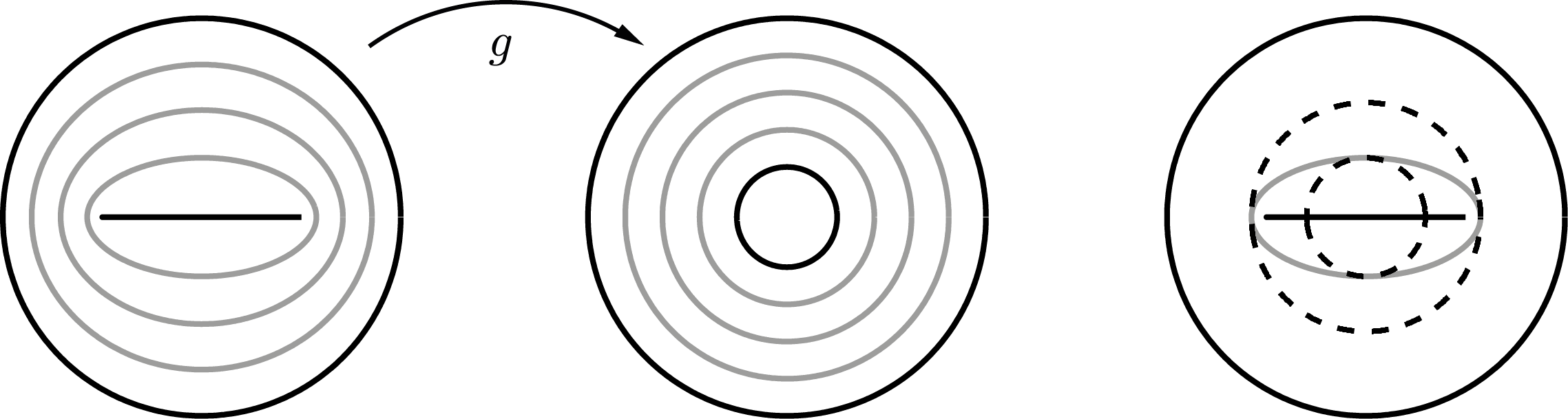}
  \caption{Mapping $g \colon \ID \setminus \{ [-a,a] \} \to \IA(r,1)$ defined in Section \ref{G example}. Roundness of an elliptical regions  $\gamma = g^{-1}(\IS(0,r))$ is illustrated on the right. }\label{Figure2}
\end{center}
\end{figure}
 
Note that $\nu_\gamma=1$ if and only if $\gamma$ is a round circle.  Further,  there is an obvious conformal invariance:  If $ \phi_b(z)$ is a M\"obius transformation of the disk,  
then $\nu(\gamma)=\nu(\phi_b(\gamma))$.  
 
\subsubsection{Roundness of germs}

Let $\mathfrak{g}=(U,\mu)$ be the germ of a quasisymmetric map $g_0:\IS\to\IS$.  We define the roundness of $\mathfrak{g}$ as follows:  Let $g:U\to \IA(r_0,1)$ with $g(\IS)=\IS$ solve the Beltrami equation (\ref{be}).  Thus $r_0={\rm mod}(g(U))$ and for all $r_0<r<1$,  $\gamma_r=g^{-1}(\IS(0,r))$ is a Jordan curve.  We set
 \begin{equation*}
 \nu_\mathfrak{g} =  \inf_{r_0\leq r < 1}\;  \nu(\gamma_r) \; \geq 1.
 \end{equation*}
  
 \subsubsection{Roundness of conformal germs}
 
 Our main interest is in conformal germs for reasons which will soon be clear.  For these germs we can give a formula for the roundness which depends completely on the boundary values alone.  
 
  \begin{theorem}\label{thm1}  Let $\mathfrak{g}=(U,0)$ be the germ of the quasisymmetric map $g_0:\IS\to\IS$.  Then
  \begin{equation*}
  \nu_\mathfrak{g} =\inf_{a\in\ID} \; \max \;  \Big\{ \, \sup_{\zeta\in\IS}(g_0\circ \phi_a)'(\zeta), \,1/\big( \inf_{\zeta\in\IS}(g_0\circ \phi_a)'(\zeta)\big)\; \Big\}.
  \end{equation*}
 \end{theorem}
 \noindent{\bf Proof.} Let $g:U\to \IA(r_0,1)$,  $r_0=e^{-{\rm mod}(U)}$,  be conformal with boundary values $g_0:\IS\to\IS$.  For $r\in (r_0,1]$ set $\gamma_r=g^{-1}(\IS(0,r))$. We extend $g$ by reflection to $g:U\cup\IS\cup U^*\to\IA(r_0,1/r_0)$. 
 The hyperbolic metric density of the annulus  $\IA_r=\IA(r,1/r)$ can be found in \cite[\S 12.2]{BeardonMinda} as
\begin{equation*}
d_{\IA_r}(z) = \frac{\pi}{2 \log(1/r)}  \; \frac{1}{|z| \cos\Big(\frac{\pi\log|z|}{2\log(1/r)}  \Big)}.
\end{equation*}
We consider the roundness of $\gamma_r$.  From the definition,  there is $a\in \ID$ and an annulus $\IA(\ell,L)$ with $\phi_a(\gamma_r)\subset \IA(\ell,L)$ and also the winding number $\omega(\phi_a(\gamma_r),0)=1$.  We set $\tilde{g}=g\circ \phi_a$.  Now,  as $\tilde{g}$ is conformal,  the ring  $\Omega = \tilde{g}^{-1}(\IA(r,1/r))$ has $\IA(L,1/L) \subset \Omega \subset \IA(\ell,1/\ell)$ and has hyperbolic density
\begin{equation}\label{hyperbolic density}
   d_\Omega(\tilde{g}^{-1}(\zeta)) |(\tilde{g}^{-1})'(\zeta)| =\frac{\pi}{2 \log(1/r)},  \hskip15pt  |\zeta|=1
\end{equation} 
 while the monotonicity property of the hyperbolic metric gives
\begin{equation}\label{hyperbolic monotonicity}
   \frac{\pi}{2 \log(1/\ell)} \leq d_\Omega(\eta) \leq  \frac{\pi}{2 \log(1/L)},  \hskip15pt  |\eta|=1
\end{equation} 
 and hence
\begin{equation*}
  \frac{\log(1/r)}{\log(1/\ell)} \leq |\tilde{g}'(\zeta)| \leq  \frac{\log(1/r)}{  \log(1/L)},  \hskip15pt  |\zeta|=1.
  \end{equation*} 
If we take the infimum over all Jordan curves $\gamma_r$ and annuli we see
$\nu_{\mathfrak{g}}^{-1} \leq |\tilde{g}'(\zeta)|  \leq   \nu_{\mathfrak{g}}$.   This is the first part of the result.
 
The converse inequality follows from an elementary first order analysis for $\tilde{g}$ on the circle.  This map is conformal in a neighbourhood of the circle and so
  \begin{equation*}
  |\tilde{g}(r\zeta)-\tilde{g}(\zeta)| =  |\tilde{g}'(\zeta)|(1-r)+ O(1-r) = |\tilde{g}'(\zeta)| \log(1/r) + O(1-r)
 \end{equation*}
 and so for $|\zeta|=1$ and $r$ close to $1$ we have with
$L_r= \sup_{|\zeta|=1} |\tilde{g}^{-1}(\zeta)|$ and $\ell_r =\inf_{|\zeta|=1} |\tilde{g}^{-1}(\zeta)|$, 
\begin{eqnarray*}
\log 1/L_r = -\log (1 - \inf_{|\zeta|=1} |\tilde{g}(r\zeta)-\tilde{g}(\zeta)|) \approx |\tilde{g}'(\zeta)|\log(1/r) \end{eqnarray*}
 and from this the result follows directly as we let $r\nearrow 1$.  Of course a similar estimate holds for $\log 1/\ell_r$.
 \hfill $\Box$
 
 \medskip
 
 We now obtain the following corollary of Lemma \ref{ext1} when applied to  $g_0\circ \phi_a$.
   \begin{corollary}\label{cor1}  Let $\mathfrak{g}=(U,0)$ be the germ of the quasisymmetric map $g_0:\IS\to\IS$.  Then $g_0$ admits a $K_G$--quasiconformal extension  $G:\ID\to\ID$ for which the distortion of $G$ satisfies
  \begin{equation*}
  K_G \leq \nu_\mathfrak{g}.
  \end{equation*}
 \end{corollary}
 We can put this another way by considering the inverse map.
 \begin{corollary}\label{corext1}
 Let $\varphi:\IA(r,1) \to \ID$ be a conformal mapping with $\varphi(\IS)=\IS$. Then $\varphi|\IS$ admits a $K_\Phi$--quasiconformal extension $\Phi:\ID\to\ID$,  $\Phi|\IS=\varphi$, with $K_\Phi \leq \nu_\varphi$.
\end{corollary}
 
 We are now able to obtain a more general result as follows. 
  
    \begin{corollary}\label{cor3}  Let $\mathfrak{g}=(U,\mu)$ be the germ of the quasisymmetric map $g_0:\IS\to\IS$,  and 
    \begin{equation*}
    K_0=\Bigg\|\frac{1+|\mu|}{1-|\mu|}\Bigg\|_{L^\infty(U)}.
    \end{equation*}
     Then $g_0$ admits a $K_G$--quasiconformal extension  $G:\ID\to\ID$ for which the distortion of $G$ satisfies
  \begin{equation*}
  K_G \leq K_0 \; \nu_\mathfrak{g}.
  \end{equation*}
 \end{corollary}
 \noindent{\bf Proof.}  Let $g:U\to\ID$ with boundary extension $g|\IS=g_0$ be the quasiconformal map determined by the germ $(U,\mu)$.  Let $\epsilon>0$ and $r$ chosen so that $\nu(g^{-1}(\IS(0,r))<\nu_{\mathfrak{g}}+\epsilon$.  Let $\varphi:\IA(s,1)\to g^{-1}(\IA(r,1))$ be a conformal map,  $\log(1/s)={\rm mod}(g^{-1}(\IA(r,1)) )$.  Now the map $g\circ \varphi: \IA(s,1) \to \IA(r,1)$ is a quasiconformal homeomorphism and can be extended by repeated reflection to a quasiconformal map $\widetilde{g\circ \varphi}:\ID\to\ID$ with the same maximal distortion as $g|g^{-1}(\IA(r,1))$.  Next, by definition $\nu_{\mathfrak{h}}$,  corresponding to the germ ${\mathfrak{h}}= (g^{-1}(\IA(r,1)),0)$ has boundary values $\varphi^{-1}$ (up to rotation).  Also,  $\nu_{\mathfrak{h}} \leq \nu (g^{-1}(\IS(0,r)))<\nu_{\mathfrak{g}}+\epsilon$,  as the roundness of $\nu_{\mathfrak{h}} $ is certainly smaller than the roundness of the inner boundary component.  Corollary \ref{cor1} implies that $\varphi^{-1}|\IS$ has an extension $\widetilde{\varphi^{-1}}:\ID\to\ID$ with distortion no more than $\nu_{\mathfrak{g}}+\epsilon$. Therefore the distortion of 
 \[ \widetilde{g\circ \varphi} \circ \widetilde{\varphi^{-1}}:\ID\to\ID \]
 is no more than $ (\nu_{\mathfrak{g}}+\epsilon)\, K_0$.  The result follows once we observe that the boundary values of this map are those of $g_0$. \hfill $\Box$
 
 \medskip
 
 We are now able to prove the following theorem which seems quite remarkable (even for conformal mappings) in light of how complicated the image of an inner boundary component might be.
 
 \begin{theorem}\label{spacefilling}  Let $f:A=\IA(r_0,1)\to \ID$ be $K_0$-quasiconformal with $f(\IS)=\IS$ and $f(0)\not\in f(A)$.  Let $L = e^{-{\rm mod}(f(A))}$ and let
  \begin{equation*}
S = \inf_{\zeta\in \IS} \; \liminf_{r\to r_0} |f(r\zeta)|,  \hskip10pt T =  \sup_{\zeta\in \IS} \; \limsup_{r\to r_0} |f(r\zeta)|  
 \end{equation*}
 Then there is a $K$-quasiconformal mapping $F:\ID\to\ID$ with $F|\IS=f$ and
 \[ K \leq \max \left\{ \frac{\log S}{\log L}, \frac{\log L}{\log T} \right\} \; K_0. \]
 If $f$ is conformal we may put $K_0=1$ and $L=r_0$.
 \end{theorem}
 \noindent{\bf Proof.}  We put $U=f(A)$ and $\mu = \mu_{f^{-1}}$ to define the germ $\mathfrak{g}=(U,\mu)$.  Now the roundness $\nu_\mathfrak{g}$ is smaller than the roundness of $\gamma_r = (f^{-1})^{-1}(\IS(r))=f(\IS(r))$ and this is rounder than the inner boundary component which lies in the annulus $\IA(S,T)$.  The condition $f(0)\not\in f(A)$ with $f(\IS)=\IS$ guarantees the winding number about $0$ is equal to $1$.  The result now follows from Corollary \ref{cor3}.\hfill $\Box$
 
 \medskip
 
 The fact that modulus increases under inclusion means that under the circumstances of the theorem $S\leq L \leq T$.  Thus if $|S-T|<\delta$ we see that for fixed $L$ our bound has the behaviour
  \[ K   \approx \left(1+\frac{\delta}{L\log 1/L}\right) \; K_0, \hskip5pt \mbox{ as $\delta\to 0$}. \] 
 However,  one might reasonably expect a $\delta^2$ term here.
 
\begin{figure}[ht!]
\begin{center}
  \includegraphics[width=.9 \textwidth]{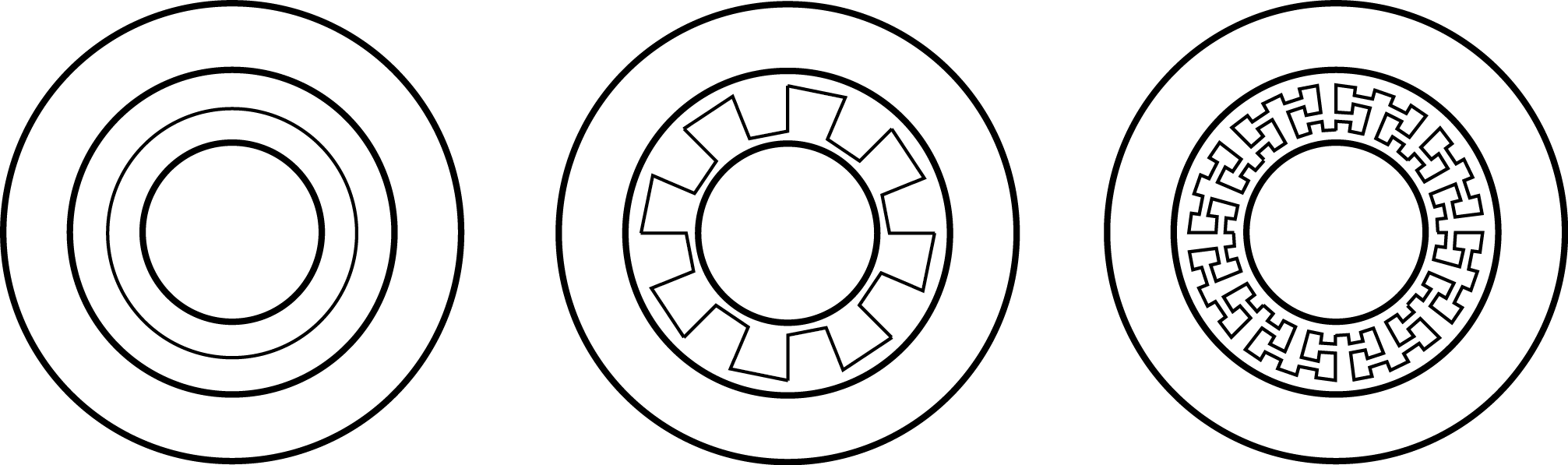}
  \caption{In Theorem \ref{spacefilling} the image of the inner boundary component can be complicated as long as the winding number about the origin is 1.  This is illustrated here with the images of $\{|z|=5/12\}$  lying in the annulus $\IA(1/3,1/2)$ and converging to a space filling curve filling that region.  There is a uniform bound $K\leq \frac{\log 3}{\log 2}$ on the distortion of the extension of any of these boundary values.}
\end{center}
\end{figure}

 \section{Roundness and modulus}
 
What we need in our applications is an estimate on the roundness of a germ $\nu_\mathfrak{g}=(U,\mu)$ in terms of ${\rm mod}(U)$ and $k_{\mathfrak{g}}=\|\mu\|_{L^{\infty}(U)}$ and a useful feature of such an estimate must be that 
\[ \nu_\mathfrak{g} \searrow \frac{1+k_{\mathfrak{g}}}{1-k_{\mathfrak{g}}},  \hskip10pt \mbox{ as ${\rm mod}(U)\to\infty$.}\] 

Our argument will give reasonable results when ${\rm mod}(U)$ is large.   We first deal with the conformal case.  Let $\ID_\rho(a,s)$ be the smallest hyperbolic disk  which contains the inner boundary component of $U$.  Such a unique disk exists because of the negative curvature of the hyperbolic plane. If we put $V=\phi_a(U)$,  then we may as well work with the germ $\mathfrak{g}=(V,0)$.   Then,  let $g:\Omega = V\cup\IS\cup V^* \to \IA(r_0,1/r_0)$ be conformal,  $g(\IS)=\IS$.  If $\log\frac{1+\ell}{1-\ell}=s$,  then $\IA(\ell,1)\subset V$. Next comparing hyperbolic densities as in \eqref{hyperbolic density} and using the monotonicity of the density with respect to domains as above \eqref{hyperbolic monotonicity}, we have for $|\zeta|=1$,
\begin{equation*}
  \frac{\log(1/L)}{\log(1/r_0)} \leq |g'(\zeta)| \leq  \frac{  \log(1/\ell)}{\log(1/r_0)} .
  \end{equation*}
Here we have set $L = \sup_{|z|=\ell} |g(z)|$.

 We now should estimate both $\ell$ and $L$ in terms of $r_0$.  For instance as $\ID(0,\ell)$ is the smallest disk (hyperbolic or euclidean) containing the inner component of $U$ we know that the diameter of this component is at most $\ell/\sqrt{2}$ and by the extremity of the Gr\"otzsch ring we have 
 \[ \log 1/r_0 = {\rm mod}(U) \leq {\rm mod} \left( R_G \left( \ell/\sqrt{2} \right) \right) .\]
 It is possible to make further estimates in this way,  but these lead to quite complicated formulas.  We give up the possibility of sharpness for a simple formula
\[
 \nu_\mathfrak{g} \leq \max_{|\zeta|=1} \left\{|g'(\zeta)|, \frac{1}{|g'(\zeta)|} \right\}  \leq \frac{\max_{|\zeta|=1} |g'(\zeta)|}{\min_{|\zeta|=1} |g'(\zeta)|} 
 \leq \frac{\log 1/r_0}{\log 1/L} =     \frac{\log 1/r_0}{\log1/r_0 - \log  L/r_0}.
\]
The components of the boundary of the image of the annulus $\IA(L,r_0)$ under the conformal mapping $g^{-1}$ both touch the circle $\{|z|=\ell\}$. A little geometry \cite[Lemma 1.3]{Martin89} and the extremality of the Teichm\"uller ring gives us the uniform estimate 
\begin{equation*}
{\rm mod}(\IA(L,r_0)) = \log L/r_0 \leq  m_T \left( \sqrt{2} \right) =  2.4984 \ldots = \beta_0.
\end{equation*}
We obtain the following lemma and its corollary (when we consider the inverse).
\begin{lemma} \label{lemma4} Let $\mathfrak{g}=(U,0)$ be a germ of a quasisymmetric mapping.  Then
\begin{equation}\label{r0est}
  \nu_\mathfrak{g} \leq \frac{{\rm mod}(U)}{{\rm mod}(U)-\beta_0}, \hskip15pt \beta_0 =  2.4984\ldots.
\end{equation} 
\end{lemma}
\begin{corollary} \label{corlemma4}
Let $\varphi:\IA(1,r_0)\to \ID$ with $\varphi(\IS)=\IS$ be conformal.  Then
\[ \nu_\varphi \leq  \frac{\log 1/r_0}{\log1/r_0 - \beta_0},  \hskip15pt \beta_0 =  2.4984\ldots.\]
\end{corollary}
The example in Section \ref{G example} below,  see (\ref{sharpexample}) suggests that $\nu_\mathfrak{g} \leq 1+8 r_0^2$ is sharp as $r_0\to 0$.  Equation (\ref{r0est}) gives 
\begin{equation*}
  \nu_\mathfrak{g} \leq \frac{1}{1-\beta_0/\log 1/r_0} \approx  1+ \frac{\beta_0}{\log 1/r_0} \hskip15pt \beta_0 =  2.4984\ldots.
\end{equation*} 
Unfortunately, Corollary \ref{corlemma4} requires $r_0\leq 0.08<e^{-\beta_0}$ to be of any use. We will next show,  as part of the proof of Theorem \ref{thm5},  how to overcome this problem.

 \section{Solutions.}
 
 In this section we develop a couple of applications of our results.   We are aware the constants $2\beta_0$ and  $4\beta_0$ in our next result can be improved.  In fact modifications of the arguments given here will do this,  but they come at a significant cost in terms of complexity, which seems pointless without obtaining a much sharper result.  We discuss sharpness in Section \ref{G example} below.
 \subsection{Problem 2.}
 \begin{theorem}\label{thm2}  Let $\mathfrak{g}=(U,\mu)$ be the germ of the quasisymmetric mapping $g_0$. Then $g_0$ admits a $K_G$--quasiconformal extension $G:\ID\to\ID$ and 
 \begin{equation}\label{result} K_G \leq    K_\mathfrak{g} \times \left\{ \begin{array}{cc}  4\beta_0/m_\mathfrak{g},   & m_\mathfrak{g}\leq  2\beta_0, \\ 1+ 2\beta_0/m_\mathfrak{g},    & m_\mathfrak{g} \geq  2\beta_0, \end{array} \right. 
 \end{equation}
 where $\beta_0=2.4984 \ldots$.
 
 The number $4\beta_0$ in (\ref{result}) cannot be replaced by any constant smaller than $1$.
 \end{theorem}
 \noindent{\bf Proof.} Let $\Omega = U \cup \IS \cup U^*$ and $g:\Omega \to \IA(r_0,1/r_0)$ be quasiconformal solving (\ref{be}) on $U$ and $g(1/\zbar)=g(z)$.   Let $f_\alpha(z) = z|z|^{\alpha-1}$, $\alpha\geq 1$,  and $\mu_\alpha = \mu_{f_\alpha\circ g }$.  Then,  as $f_\alpha|\IS$ is the identity,  we see that $\mathfrak{g}_\alpha = (U,\mu_\alpha)$ is another germ for $g_0$.  Let $\nu(z) = \mu_\alpha(z),  \;\; z\in U$ and $0$ otherwise.  Integrate this Beltrami coefficient to a quasiconformal mapping $h:\ID\to\ID$,  $\mu_h=\nu$.  Now as both $h$ and $f_\alpha\circ g$ solve the same Beltrami equation on $U$ there is a conformal mapping $\varphi:h(U)\to (f_\alpha\circ g)(U) = \IA(r_{0}^{\alpha},1)$ so that $(f_\alpha\circ g)(z)=\varphi\circ h(z)$.  Since $\varphi(\IS)=\IS$ we see $(h(U),0)$ is the germ of $\varphi|\IS$ and ${\rm mod}(h(U))=\log 1/r_{0}^{\alpha}$.  Then,  combining Lemma \ref{lemma4} and Corollary \ref{cor1} we see that $\varphi|\IS$ has an extension to a quasiconformal $\Phi:\ID\to\ID$ with distortion at most
 \[ K_\Phi \leq  \frac{\log 1/r_{0}^{\alpha}}{\log 1/r_{0}^{\alpha} - \beta_0} \]
 provided the right-hand side is positive.  Then  $G=\Phi\circ h$ is the extension we seek,  $G |\IS = g_0$ and as $K_{f_\alpha}=\alpha \geq 1$
 \[ K_G \leq \alpha \, K_\mathfrak{g}  \frac{\log 1/r_{0}^{\alpha}}{\log 1/r_{0}^{\alpha} - \beta_0} = K_\mathfrak{g}  \frac{\alpha^2 \, \log 1/r_{0}}{\alpha \, \log 1/r_{0}- \beta_0}. \] 
 We need to make a good choice of $\alpha$ here.  If $2\beta_0\geq \log 1/r_0$,  then the minimum occurs,  otherwise we put $\alpha=1$ and obtain 
 \[ K_G \leq   K_\mathfrak{g} \times \left\{ \begin{array}{cc} \frac{4\beta_0} { \log 1/r_{0}}, & r_0\geq e^{-2\beta_0}, \\ \frac{1}{1- \beta_0/\log(1/r_0)}, & r_0\leq e^{-2\beta_0} \end{array} \right. \] 
 and a little analysis gives the result at (\ref{result}) once we note that,  by the definition at (\ref{2}), $m_\mathfrak{g}=\log1/r_0$. 
 
 \medskip
 
 The example given by consideration of the germ $(\ID\setminus[-a,a],0)$ at (\ref{sharpXX}) establishes the claim that $4\beta_0$ cannot be replaced by any  number less than 1.  \hfill $\Box$
 
 \bigskip
 
We may interpret this result as follows by considering the inverse mappings.

\begin{theorem}  \label{thm5} Let $f:\ID\to\ID$ be a homeomorphism of finite distortion with $|f|\to 1$ as $|z|\to1$ and define   
  \begin{equation*}
 K(r) = \| K(z,f)\|_{L^\infty(\ID \setminus \ID(0,r))}.
 \end{equation*} Then the boundary values $f_0=f|\IS:\IS\to\IS$ exist and have a $K$--quasiconformal extension with
 \begin{equation*}
 K \leq \left(   1+ \frac{4\beta_0}{\log \frac{1}{r_0}} \right) \;  K(r), \hskip30pt \beta_0=2.4984 \ldots,
 \end{equation*}
 and 
 \[ \log 1/r_0 = {\rm mod}(f(\IA(r,1)) \geq \frac{ \log 1/r }{K(r) }. \]
 \end{theorem}
 In fact,  contained in the proof of Theorem \ref{thm2} is the estimate on roundness of a conformal germ $\mathfrak{g}$,
 \begin{equation*}
   \nu_{\mathfrak{g}} \leq 1+ \frac{4\beta_0}{\log \frac{1}{r_0}},  \hskip10pt \mbox{as $r_0\to 1$}.
 \end{equation*}
 Then the example at (\ref{sharpexample}) shows that this is the correct form  and that $4\beta_0$ cannot be replaced by any constant less than $1$.  
 
 \medskip
 
 We are now in a position to consider the asymptotically conformal extension at Lemma \ref{lemma3}.  We recall that if $f'$ is Lipschitz and if $G(re^{i\theta}) = r^{f'(\theta)} e^{if(\theta)}$, then
\begin{equation*}
\big|\,\mu_G(re^{i\theta})\,\big| = \left( 1+\left( \frac{2f'(\theta)}{f''(\theta)\log r}\right)^2\right)^{-1/2}. \end{equation*}
  We therefore have 
  \[ K(r) = \frac{\sqrt{(f''/2f')^2 \log^2 r + 1}+|f''/2f'| \; \log 1/r}{\sqrt{(f''/2f')^2 \log^2 r + 1}-|f''/2f'| \; \log 1/r} \leq \frac{\sqrt{a_{f}^{2} \log^2 r + 1}+a_{f}\, \log 1/r}{\sqrt{a_{f}^{2} \log^2 r + 1}-a_{f} \, \log 1/r} \]
with $a_f=\sup_\theta |f''/2f'|$. From the formula for $G$ we see that in Theorem \ref{thm5} $\log 1/r_0 \geq (\min_\theta f' ) \, \log 1/r$,  so there is an extension with $b_f =1/ \min_\theta f' $, 
\[ K\leq  \left(   1+ \frac{4\beta_0 \, b_f }{\log \frac{1}{r}} \right) \;  K(r). \] 
  One can explicitly solve the associated minimisation problem for $\log 1/r$,  but again it is very complicated.  We can estimate
    \[ K(r) \leq (1+2 a_{f} \, \log 1/r)^2.  \]
  Then some asymptotic analysis gives the following theorem.
  \begin{theorem}\label{Lf}  Let $g= e^{i f(\theta)} :\IS\to\IS$ have continuous second derivatives.  Let
  \begin{eqnarray*}
A_f =  2\beta_0\;  \frac{\max_{\theta} \; f''(\theta)/f'(\theta)}{\min_{\theta} \;f'(\theta)}, \quad \beta_0=2.4984 \ldots.  
  \end{eqnarray*}
  Then $g$ has a quasiconformal extension $G:\ID\to\ID$ with 
  \[ K_G \leq 1+ 4\sqrt{A_f}+9A_f. \]
  \end{theorem} 
  It remains an interesting problem to formulate a sharper and M\"obius invariant version of this theorem.

\subsection{Problem 1}   We can now apply the solution of Problem 2 to solve Problem 1 when we note in this case that $m_{\mathfrak{g}}={\rm mod}(U)$.
 
  \begin{theorem}\label{mainthm2}  Let $\mathfrak{g}=(U,0)$ be the germ of the quasisymmetric mapping $g_0$. Then $g_0$ admits a $K_G$--quasiconformal extension $G:\ID\to\ID$ and 
 \begin{equation}\label{result2X} 
   K_G \leq    \left\{ \begin{array}{cc}  4\beta_0/{\rm mod}(U),   &  \mbox{ if $\;\;2\beta_0 \leq {\rm mod}(U)$},  \\ 1+ 2\beta_0/{\rm mod}(U),    &\mbox{ if $\;\; 2\beta_0 \geq {\rm mod}(U)$}, \end{array} \right. 
 \end{equation}
 where $\beta_0=2.4984 \ldots$.
 
  The number $4\beta_0$ in (\ref{result2X}) cannot be replaced by any constant smaller than $1$.
 \end{theorem}

\section{Separation in modulus} 
    We can now use this result to piece together a more general result.  Let $\Omega\subset \IC$ be a domain.  We say that a set $E\subset \Omega$ can be $Q$-separated in modulus if there is  a countable collection of disjoint annular regions $A_i \subset \Omega$ with $\Mod(A_i)\geq Q$
    such that  
    \[ E \subset \bigcup_i D_i \]
    and $D_i\subset \Omega $ is the bounded component of $\IC\setminus A_i$.
    
    This condition seems not too far from the definition of uniformly perfectness of a set when described in terms of modulus,  see for  instance \cite{JV}, but it is different (as is easily seen in the way one might agglomerate components).  Further,  it is easy to construct such sets $E\subset \ID$ which are  $Q$-separated in modulus with $\overline{E}=\IS$.  Indeed one of our subsequent applications to Riemann surfaces will exhibit this property.
    
    We begin with two lemmas.
    
    \begin{lemma} \label{lemma5}Let  $E,E'\subset \ID$ be compact and connected, $\beta_0=2.4984 \ldots$ and $f:\ID\setminus E\to\ID\setminus E'$ be a quasiconformal homeomorphism with $|\mu_f| \leq k<1$ and $f(\IS)=\IS$.  Then $f|\IS$ admits a $K_F$--quasiconformal extension $F:\ID\to\ID$ with
   \[
    K_F \leq K \left( 1+ \frac{4\beta_0}{Q} \right) \left( 1+\frac{4\beta_0}{Q'} \right) ,  \hskip15pt  K=\frac{1+k}{1-k},\;\;  Q={\rm mod}(\ID\setminus E),\;\; Q'={\rm mod}(\ID\setminus E').
\]
\end{lemma}
\begin{proof}  The boundary values of the conformal mapping $\varphi:\ID\setminus F\to\IA(e^{-Q'},1)$ have a $(1+4\beta_0/Q')$ extension $\Phi$ to the disk.  Then $\varphi\circ f:\ID\setminus E \to \IA(e^{-Q},1)$ admits a $(1+4\beta_0/Q)K$ extension $\widetilde{\varphi\circ f}$ from Theorem \ref{thm2}.  Then $F=\Phi^{-1} \circ \widetilde{\varphi\circ f}:\ID\to\ID$ suffices.
\end{proof}

\begin{lemma} \label{lemma6} Let  $A,B\subset \IC$ be doubly connected domains with Jordan outer boundary components $\partial_+A$ and $\partial_+B$ bounding disks $\Omega_A$ and $\Omega_B$ respectively and $\beta_0=2.4984 \ldots$. Let $f:A \to B$ be a quasiconformal homeomorphism with $|\mu_f| \leq k<1$.   Then $f$ extends to a map $f_0:\partial \Omega_A\to \partial \Omega_B$ which admits a $K_F$--quasiconformal extension $F:\overline{\Omega_A}\to\overline{\Omega_B}$ with
   \[
    K_F\leq K \left(1+ \frac{4\beta_0}{Q}\right) \left( 1+\frac{4\beta_0}{Q'}\right),  \hskip15pt  K=\frac{1+k}{1-k},\;\;  Q={\rm mod}(A),\;\; Q'={\rm mod}(B).
\]
    \end{lemma}
    \noindent{\bf Proof.}  Take  Riemann mappings $\Omega_A\mapsto \ID$ and $\Omega_B\mapsto \ID$ to reduce the problem to Lemma \ref{lemma5}.  The result follows. \hfill $\Box$
    
\medskip
 
\begin{figure}[ht!]
\begin{center}
  \includegraphics[width= \textwidth]{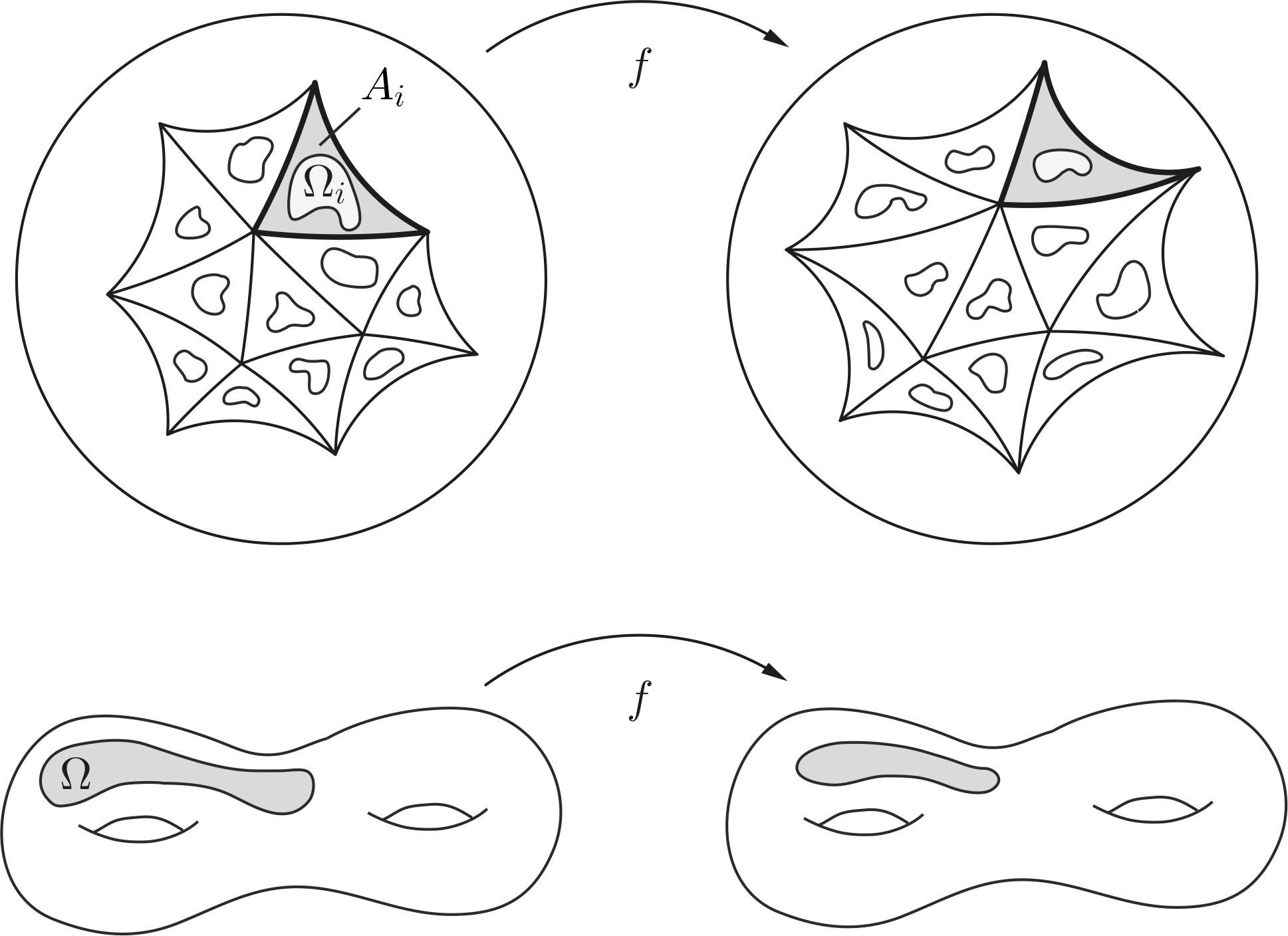}
  \caption{Theorems \ref{sepinmod} and \ref{Riemann surface}.}
\end{center}
\end{figure}
    
    The next theorem is the general result we seek.
    
    \begin{theorem}\label{sepinmod}  Let $\Omega \subset \IC$ be a Jordan domain and let $f:\Omega\to f(\Omega)\subset \IC $ be a homeomorphism of finite distortion with $f(\Omega)$ also a Jordan domain.  Suppose that $E\subset \Omega$ is a set which is  $Q$-separated in modulus  with
 \begin{equation}\label{42}\|\; \mu_f \; \|_{L^\infty(\Omega \setminus E)} \leq k < 1.
 \end{equation}
 Then $f|\partial\Omega$ has a $K_F$--quasiconformal extension $F:\overline{\Omega}\to\overline{f(\Omega)}$ and
 \begin{equation}\label{43}
 K_F \leq \left( 1+\frac{4 \beta_0}{Q} \right) \left( 1+\frac{4 \beta_0 K_0}{Q} \right) \; K_0, \hskip20pt K_0=\frac{1+k}{1-k}, \quad \beta_0=2.4984 \ldots.
 \end{equation} 
    \end{theorem}
    
    Note that the requirement that $\Omega$ and $f(\Omega)$ be Jordan domains is simply to avoid tedious problems with the definition of boundary values through prime ends and the like.
    
    \medskip
    
    \noindent{\bf Proof.}  Let $\{A_i\}_{i=1}^{\infty}$ be the annuli separating $E$ given from the definition of  $Q$-separated in modulus. We may assume by an approximation argument (choosing sub-annuli for some $Q'<Q$) that for each $i$,  $\Omega_i=\overline{A_i\cup D_i}$ (where $D_i$ is the bounded component of $\IC\setminus A_i$) is a smoothly bounded Jordan disk and that the distortion bound \eqref{42} holds on a neighbourhood of $\partial A_i$. It follows that $f(\Omega_i)$ is a quasidisk.   Lemma \ref{lemma6} gives $F_i:\Omega_i\to f(\Omega_i)$ with the stated distortion bounds as $Q'={\rm mod}(f(A_i)) \geq {\rm mod}(A_i)/K_0 = Q/K_0$. We define a new mapping by 
    \[ F(z) = \left\{\begin{array}{ll}  f(z), & z\in \Omega \setminus \bigcup_{i=1}^{\infty} A_i \cup D_i, \\
F_i(z),  & z\in  \overline{A_i \cup D_i}. \end{array}\right.\]
    It is a moments work to see that $F$ is $K_F$--quasiconformal and $K_F$ satisfies \eqref{43}. \hfill $\Box$

\section{Examples} \label{G example}
 
 Here we present an elementary example to show that some conditions are necessary of the set $E$ where we relax control of the distortion.
 
We consider first what could happen, if the mapping $g_0 \colon \IS \to \IS$ was not quasisymmetric. Instead of the disk we work in the upper-half space.  The map 
\[ f_0(x) = \left\{ \begin{array}{ll} x, & x \leq 0, \\ x^2, & x\geq 0 \end{array} \right. \]
is not quasisymmetric.  Let $h:[0,\pi]\to [0,1]$ be a homeomorphism and set
\[ F(z) =  z |z|^{1-h(\theta)}, \hskip20pt \theta=\arg(z).\]
  Then $F|\IR = f_0$.  For suitable choices of $h$ we can arrange that large distortion is supported on a thin wedge. For instance if $\epsilon>0$ is given
  \[ h(\theta) = \left\{\begin{array}{ll} 0, & \theta<\pi/2-\epsilon, \\ (\theta-\pi/2+\epsilon)/2\epsilon, &  \pi/2-\epsilon\leq \theta\leq \pi/2+\epsilon,  \\ 1, & \theta>\pi/2+\epsilon \end{array}\right. \]
  then 
    \[ K(z, F) = \left\{\begin{array}{ll} 1, & \arg(z) > \pi/2 +\epsilon, \\ 2, & \arg(z) < \pi/2 -\epsilon. \end{array}\right.
    \]
  It would be interesting to determine how finer conditions,  such as if the distortion is bounded outside a cusp (with endpoint on $\IR$) might influence the regularity of the boundary values.

\bigskip

The upper bound for $\nu_\mathfrak{g}$ obtained in Lemma \ref{lemma4} is not sharp. Next we consider an example to get an idea what the sharp bound might be. For $a \in (0,1)$ we define a mapping $g \colon \ID \setminus \{ [-a,a] \} \to \IA(r_0,1)$ as a composition of three mappings $g = g_3 \circ g_2 \circ g_1$.  Here $g_3(z) = r_0 (z-\sqrt{z^2-1})$, $g_2(z) = \sin z$ and
\[
  g_1(z) = \frac{\pi}{2 {\cal K}(a^2)} \textrm{arcsn}\left( \frac{z}{a},a^2 \right),
\]
where \textrm{arcsn} is the inverse Jacobi elliptic sine function. The mapping $g$ is plotted in Figure \ref{Figure2}. The curves $g^{-1}(\IS(0,r))$ are hyperbolic ellipses \cite[Theorem 3.5]{ALV}
\begin{equation*}
  \{ z \in \ID \colon \rho_\ID(-a,z)+\rho_\ID(a,z) = c \}
\end{equation*}
and it is easy to show that
\begin{equation}\label{sharpexample}
  \nu_\mathfrak{g} = \frac{1+a^2}{1-a^2}.
\end{equation} 
The radius $r_0$ is defined by $a$ and can be solved from the equation $g(1)=1$. We cannot  solve this explicitly, but $r_0 \approx a/2$ when $a$ is close to 0 and $r_0 \approx a$ when $a$ is close to 1.

\medskip

Next note that the mappings $g_1$ and $g_2$ are conformal and that the map $g_3:g_2\circ g_1(\IS) \to \IS$ from the ellipse $g_2\circ g_1(\IS)$ are the boundary values of a linear mapping.  In fact 
\[ g_{3}^{-1} (\zeta) =  \frac{1}{2} \big( \frac{1}{r_0} \zeta + {r_0}\bar\zeta\big), \hskip10pt |\zeta|=1. \]
The best quasiconformal extension of these boundary values is the linear map $g_{3}^{-1} (w) = \frac{1}{2}\big( \frac{1}{r_0} w+ {r_0}\bar w \big)$ itself,  with distortion $K=\frac{1+r_{0}^2}{1-r_{0}^{2}}$,  \cite{RS,Str}.  The next theorem follows.

\begin{theorem}  Let $\mathfrak{g}=(\ID\setminus [-a,a],0)$ with associated quasisymmetric mapping $g_0:\IS\to\IS$.  Then any quasiconformal extension of $g_0$ has distortion   
\begin{equation*}
K \geq \frac{1+r_0^2}{1-r_0^2} =   \frac{1+e^{-2m_\mathfrak{g}}}{1-e^{-2m_\mathfrak{g}}} ,   \hskip15pt \log \frac{1}{r_0} = m_\mathfrak{g}.
\end{equation*}
\end{theorem}
As $m_{\mathfrak{g}}\to 0$,
\begin{equation}\label{sharpXX}
 \frac{1+e^{-2m_\mathfrak{g}}}{1-e^{-2m_\mathfrak{g}}} \geq \frac{1}{m_\mathfrak{g} }
 \end{equation} 
and this establishes the sharpness claimed in the solutions to Problems 1 and 2.
 \section{Distortion and topology}
 
 The aim of this section is to apply the results we have found above to show that in a degenerating sequence of Riemann surfaces the blowing up of the distortion of a reference map from a base surface cannot be confined to a simply connected set.  We will establish this result with explicit estimates.  These need some concepts which we now develop.
 
  \medskip
 
 Let $\Sigma$ be a hyperbolic Riemann surface and  ${\Omega}$ a   {\em Jordan disk} in $\Sigma$.  Let $U$ be a simply connected subset of $\Sigma$ with $\overline{\Omega}\subset U$.  Then the ring $U\setminus \overline{\Omega}$ is conformally equivalent to an annulus $\IA(r,1)$,   since $\Sigma$ is hyperbolic   $0<r<1$,  and the modulus of   $U\setminus \overline{\Omega}$ is defined to be $\Mod(U\setminus \overline{\Omega}) = \log\frac{1}{r}$.  We then define the modulus of $\Omega$ in $\Sigma$ as 
 \begin{equation*}
 \Mod_\Sigma(\Omega) = \sup_U \; \{\, \Mod(U\setminus \overline{\Omega})\, \}
 \end{equation*}
where the supremum is over simply connected subsets of $\Omega$. It is easy to see that if $\Omega$ is a Jordan domain in $\Sigma$, then $\partial \Omega$ is locally connected and an elementary continuity argument implies there is a simply connected set $U$ with $\partial\Omega\subset U$ and hence $0< \Mod_\Sigma(\Omega)<\infty$.  While it is generally impossible to identify this number, one can estimate it by considering the hyperbolic distance of $\partial \Omega$ to a set of arcs which cut $\Sigma$ into a simply connected region and which do not meet $\overline{\Omega}$.
 
 Next,  let $\Gamma$ be the universal covering group of M\"obius transformations of the disk for $\Sigma$. So
 \[ \Sigma = \ID/\Gamma. \]
 If $U\subset \Sigma$ is simply connected and $\Omega\subset U$, then $U$ lifts to  (more correctly a lift can be chosen so that)  a disjoint collection of simply connected sets $\{U_\gamma:\gamma\in\Gamma\}$ containing the lifts $\{\Omega_\gamma:\gamma\in\Gamma\}$ of $\Omega$.  Thus,  given a Jordan domain $\Omega$, we lift to the universal cover to see a disjoint collection of simply connected $U_\gamma$ with $\overline{\Omega_\gamma} \subset U_\gamma$ and because the projection is locally conformal, and conformal as a map $U_\gamma\to U$ for each $\gamma\in \Gamma$,  we have 
 \begin{equation*}
 \Mod(U_\gamma\setminus \overline{\Omega}_\gamma) = \Mod_\Sigma(\Omega) > 0.
 \end{equation*}
In particular we see that 
\[ E = \cup_{\gamma\in \Gamma} \; \Omega_\gamma\]
can be $Q=\Mod_\Sigma(\Omega)$ separated in modulus in $\ID$.

\medskip

 Next suppose that $\tilde{\Sigma}$ is a Riemann surface with covering group $\tilde\Gamma$and that $f:\Sigma \to \tilde\Sigma$ is a homeomorphism of finite distortion.  In what follows the supposition that $f$ is a homeomorphism can be weakened. All that is really required is that  $f$ is a mapping of finite distortion which is a homotopy equivalence (or possibly just $\pi_1$-injective).  However this leads to a number of technical difficulties which we do not wish to go in to.  The mapping $f$ lifts to the universal cover to a homeomorphic map $F:\ID\to\ID$ of finite distortion, automorphic with respect to these groups:  $F\circ \gamma = \tilde\gamma\circ F$, where $\gamma\mapsto\tilde\gamma$ is the isomorphism between fundamental groups induced by the map $f$.  Local conformal coordinates define the $z$ and $\zbar$ derivatives and we can define the Beltrami coefficient in the usual way.  Next suppose that
 \begin{equation*}
 \| \; \mu_f \;\|_{L^\infty(\Sigma\setminus \Omega)} \leq k < 1.
 \end{equation*}
 Then,  by construction, for each $\gamma$, 
 \begin{equation*}
  \| \; \mu_F \;\|_{L^\infty(U\gamma \setminus \Omega_\gamma)} \leq k < 1.
 \end{equation*}
 Now fix $\gamma_0\in \Gamma$,  set $K=\frac{1+k}{1-k}$, $U_0=U_{\gamma_0}$  and  assume that $\partial U_0$ is a Jordan curve ($U_0$ is simply connected) by an approximation argument.  Now,  by Lemma \ref{lemma6} we know that   $F_0= F|\partial U_0$ admits a  $K_\gamma$--quasiconformal extension $\tilde{F}_0:\overline{U_0} \to \overline{F(U_0)}$ with $\tilde{F}_0|\partial U_0 = F_0|\partial U_0 = F|\partial U_{\gamma_0}$ and 
 \[ K_0 \leq K \; (1+4\beta_0/{\rm mod}_\Sigma(\Omega)) (1+4\beta_0 K /{\rm mod}_\Sigma(\Omega)). \]
 We can therefore define a new map using the fact that $F$ is automorphic and we have only changed $F$ on part of a fundamental domain.   Thus
 \begin{equation*}
 \tilde{F}(z) = \left\{\begin{array}{ll} \tilde{\eta}^{-1}\circ \tilde{F}_0 \circ  \eta (z),   & \mbox{if $z\in U_\gamma$ and $\eta(U_\gamma) = U_{\gamma_0}$}, \\ F(z), & \mbox{otherwise.} \end{array} \right.
 \end{equation*}
This new map satisfies the same distortion bounds as $\eta$ and $\tilde{\eta}$ are M\"obius,  is automorphic with respect to the groups $\Gamma$ and $\tilde{\Gamma}$ and therefore decends to a map $\Sigma\to\tilde{\Sigma}$.  We have proved the following theorem.
 
  \begin{theorem}\label{Riemann surface}  Let $\Sigma$ be a  hyperbolic Riemann surface and $\Omega$ a Jordan domain in $\Sigma$.  Suppose that $\tilde{\Sigma}$ is another Riemann surface and that $f:\Sigma \to \tilde{\Sigma}$ is a mapping of finite distortion with
 \begin{equation*}
 \|\mu_f \|_{L^\infty(\Sigma\setminus\Omega)} = k < 1, \hskip10pt K=\frac{1+k}{1-k}.
 \end{equation*}
 Then there is a $K^*$--quasiconformal map $ f^*:\Sigma\to\tilde{\Sigma}$ homotopic to $f$ and 
 \[ K^* \leq  K \; (1+4\beta_0/{\rm mod}_\Sigma(\Omega)) (1+4\beta_0 K /{\rm mod}_\Sigma(\Omega)), \quad \beta_0=2.4984 \ldots.  \]
 \end{theorem}
 This theorem quantifies the well known fact that in a sequence of degenerating Riemann surfaces,  there is an essential loop on which the distortion back to a reference surface is blowing up.
An interesting thing to note here is that $K^* \to K$ as ${\rm mod}_\Sigma(\Omega)\to \infty$,  and that typically $K^* \leq 8 \beta_0 \;  K^2/{\rm mod}_{\Sigma(\Omega)}^{2}$.

\medskip

Finally,  there are obvious extensions of this theorem to the case that $\Omega$ is a disjoint union of Jordan domains and similar estimates will pertain.

\bigskip

\noindent R. Kl\'en -  University of Turku, Finland, and Massey University,  Auckland, New Zealand, riku.klen@utu.fi

\noindent G. J. Martin -  Massey University,  Auckland, New Zealand,    g.j.martin@massey.ac.nz

\end{document}